\documentclass[11pt]{amsart}
\usepackage{amsmath, amsfonts, amscd, latexsym, amsthm, amssymb}
\usepackage{hyperref}

\def \c{\mathbb{C}}
\def \z{\mathbb{Z}}
\def \r{\mathbb{R}}
\def \n{\mathbb{N}}

\def \K{\mathcal{K}}
\def \P{\mathcal{P}}
\def \R{\mathcal{R}}
\def \L{\mathcal{L}}

\def \Gr{\textup{Gr}}
\def \GL{\textup{GL}}
\def \SL{\textup{SL}}
\def \SO{\textup{SO}}
\def \SP{\textup{SP}}

\def \Spec{\textup{Spec}}
\def \Vol{\textup{Vol}}

\theoremstyle{plain}
\newtheorem{Th}{Theorem}[section]
\newtheorem{Lem}[Th]{Lemma}
\newtheorem{Prop}[Th]{Proposition}
\newtheorem{Cor}[Th]{Corollary}

\theoremstyle{definition}

\newtheorem{Def}[Th]{Definition}
\newtheorem{Rem}[Th]{Remark}

\begin{document}
\title[Semigroup of representations and Kazarnovskii's theorem]
{Moment polytopes, semigroup of representations
and Kazarnovskii's theorem}
\author{Kiumars Kaveh, A. G. Khovanskii}
\date{\today}

\maketitle

\begin{center}
{\it To Stephen Smale, our mathematical hero}
\end{center}

\begin{abstract}
Two representations of a reductive group $G$ are spectrally
equivalent if the same irreducible representations appear in both of
them. The semigroup of finite dimensional representations of $G$
with tensor product and up to spectral equivalence is a rather
complicated object. We show that the Grothendieck group
of this semigroup is more tractable and give a description of it in
terms of moment polytopes of representations. As a corollary, we give a proof of the
Kazarnovskii theorem on the number of solutions in $G$ of a system
$f_1(x) = \cdots = f_m(x) = 0$ where $m=\dim(G)$ and
each $f_i$ is a generic function in the space of matrix elements of a representation
$\pi_i$ of $G$.\\
\end{abstract}



\section{Introduction}
To each commutative semigroup $S$ one associates its Grothendieck semigroup $\Gr(S)$
which is a semigroup with cancelation.
We say that two elements $a, b \in S$ are {\it analogous} and write $a \sim b$ if there is $c \in S$ with
$a+c = b+c$ (where we write the semigroup operation additively).
The relation $\sim$ is an equivalence relation and respects the addition.
The set of equivalence classes of $\sim$ together with the induced addition is the
{\it Grothendieck semigroup of $S$} denoted by $\Gr(S)$. There is a natural homomorphism
$\rho: S \to \Gr(S)$ which associates to each element its equivalence class.

The semigroup $\Gr(S)$ has the {\it cancelation property}, i.e. if for
$a, b, c \in \Gr(S)$ we have $a+ c = b+c$ then $a = b$. Moreover,
for any homomorphism $\varphi: S \to H$ where $H$ is a semigroup with cancelation, there exists a
unique homomorphism $\bar{\varphi}: \Gr(S) \to H$ such that $\varphi = \bar{\varphi} \circ \rho$.
In particular, under the homomorphism $\varphi$, analogous elements have the same image.

Any semigroup $H$ with cancelation naturally extends to a group, namely its {\it group of formal differences} which
consists of pairs of elements from $H$ where we consider two pairs $(a, b)$ and $(c,d)$ equal
if $a+d = b+c$. The {\it Grothendieck group of a semigroup $S$} is the group of formal differences of $\Gr(S)$.

The Grothendieck semigroup of $S$ contains significant information about $S$ and often is
simpler to describe than $S$ itself.

In this note we will discuss a description of analogous elements,
the Grothendieck semigroup and the homomorphism $\rho: S \to \Gr(S)$ for the following
two examples of semigroups. The first one is the motivating example for the second
one which is the main contribution of the present paper.\\

\noindent{\bf Example 1:} Let $\K$ be the semigroup of nonempty finite subsets in the
lattice $\z^n$ with respect to the addition of subsets. The semigroup
$\K$  is rather complicated, but its Grothendieck semigroup is easy to describe.
It is isomorphic to the semigroup $\P$ of convex integral
polytopes in $\r^n$ with respect to addition (also called the Minkowski sum).
The homomorphism $\rho$ sends a finite subset $A \subset \z^n$ to its convex hull
$\Delta (A)$.

From the description of the semigroup $\K$ one obtains a simple proof of
the Bernstein-Ku\v{s}nirenko theorem.
The support of a Laurent polynomial $f$ is the finite set of $\alpha \in \z^n$
where $x^\alpha$ appears in $f$ with nonzero coefficient. For a finite nonempty set $A \subset
\z^n$ let $L_A$ denote the subspace of Laurent polynomials whose supports lie in $A$.
Let $A_1, \ldots, A_n \subset \z^n$ be finite nonempty subsets with convex hulls
$\Delta_1, \ldots, \Delta_n$ respectively. The
Bernstein-Ku\v{s}nirenko theorem states that the number of solutions
in $(\c^*)^n$ of a system $f_1(x)=\dots=f_n(x)=0$, where each $f_i$ is a  generic
Laurent  polynomial in $L_{A_i}$, is equal to the mixed volume of the polytopes
$\Delta_1, \ldots, \Delta_n$ multiplied by $n!$
(Sections \ref{sec-Bern-Kush} and \ref{sec-matrix-elements-torus}).\\

\noindent{\bf Example 2:} Let $G$ be a complex connected reductive
algebraic group of dimension $m$. We say that two finite dimensional
representations $\pi_1$, $\pi_2$ of $G$ are {\it spectrally
equivalent} if they have the same $G$-spectrum, i.e. the same
irreducible representations appear in both of them but perhaps with
different multiplicities. Let $\R_{\Spec}(G)$ be the semigroup of
finite dimensional representations of $G$ with respect to tensor
product and up to spectral equivalence. This semigroup is quite
complicated. In Section \ref{sec-main-spectrally-equiv} we describe
the Grothendieck semigroup of this semigroup: let $\z^n$ be the
weight lattice of $G$, i.e. the lattice of characters of a maximal
torus of $G$, and let $\r^n$ be its real span. The Weyl group $W$ of
$G$ is a finite group generated by reflections acting on $\r^n$
preserving the lattice $\z^n$. One fixes a fundamental domain $C$
for the action of $W$ and call it the positive Weyl chamber. Every
finite dimensional irreducible representation $\pi$ is determined by
a finite number of so-called highest weights which are
integral points in the positive Weyl chamber. The convex hull of the
union of the $W$-orbits of these highest weights is called the {\it
weight polytope} of the representation denoted by $\Delta_W(\pi)$.
It is an integral $W$-invariant polytope in $\r^n$.
Also we call the intersection of $\Delta_W(\pi)$ with the positive Weyl
chamber, the {\it moment polytope} of $\pi$ and denote it by $\Delta_W^+(\pi)$.
The main result is Theorem \ref{th-20} which states that the Grothendieck
semigroup of $\R_{\Spec}(G)$ is isomorphic to the semigroup $\P_W$ of integral
$W$-invariant polytopes in $\r^n$ together with the Minkowski sum of
polytopes. The homomorphism $\rho$ sends a representation $\pi$ to
its weight polytope $\Delta_W(\pi)$. Equivalently, one can describe the
Grothendieck semigroup of $\R_{\Spec}(G)$ in terms of the moment polytopes
(Theorem \ref{th-21}).

As in the Bernstein-Ku\v{s}nirenko theorem, Example 2 then allows
us to obtain a simple proof of the Kazarnovskii theorem.
For a representation $\pi$ of $G$ let $L_\pi$ denote its space of matrix elements,
i.e. the subspace spanned by the matrix entries (in some basis) of the representation $\pi$
regarded as regular functions on $G$.
Let $\pi_1, \ldots, \pi_m$ be finite dimensional representations of $G$ with
moment polytopes $\Delta_1=\Delta^+_W(\pi_1), \ldots, \Delta_m=\Delta^+_W(\pi_m)$
respectively. The Kazarnovskii theorem computes the
number of solutions of a system $f_1(x)=\dots=f_m(x)=0$ on the group $G$,
where each $f_i$ is a generic function from the space of matrix
elements $L_{\pi_i}$, in terms of the polytopes $\Delta_i$.
(Sections \ref{sec-Bern-Kush} and \ref{sec-matrix-elements-reductive}).

In Section \ref{sec-G-C} we rewrite the Kazarnovskii formula as the
mixed volume of certain polytopes constructed out of moment polytopes and the
so-called Gelfand-Cetlin polytopes.
To each irreducible representation of a classical group there corresponds its
{\it Gelfand-Cetlin polytope}. Given a representation $\pi$ of a
classical group $G$ one defines the polytope $\tilde{\Delta}(\pi)$ to be the polytope
fibred over the polytope $\Delta_W^+(\pi) = \Delta_W(\pi) \cap C$ and with Gelfand-Cetlin polytopes as fibres.
For the classical groups, the Kazarnovskii theorem can be formulated exactly as the Bernstein-Ku\v{s}nirenko
theorem: the number of solutions of the system in the theorem is equal to the
mixed volume of the polytopes $\tilde{\Delta}(\pi_i)$ multiplied by $m!$ .

The main theorem (Theorem \ref{th-20}) was conjectured by the second
author in the early $90$'s after the Kazarnovskii theorem had appeared (\cite{Kazarnovskii}).
We managed to prove it using the PRV conjecture/theorem which is a deep fact
about tensor product of irreducible representations (Theorem \ref{th-18}).

The weight polytope (or moment polytope) of a representation also plays important role in questions
related to the geometry of the group $G$, its compactifications and its subvarieties. For
some interesting results in this direction see \cite{Kapranov, Valentina1, Valentina2, Timashev}.

The proof of Bernstein-Ku\v{s}nirenko theorem in this note, with some improvements,
is the same as the one in \cite{Askold-sum-of-finite-sets}.
The main theorem (Theorem \ref{th-20}) allows us to extend this argument
to general reductive groups and prove the Kazarnovskii theorem.
Another ingredient in the proof is our intersection theory of
finite dimensional subspaces of rational functions on algebraic varieties (developed in \cite{Askold-Kiumars-MMJ}).
This is briefly reviewed in Section \ref{sec-int-index}.
A crucial step in our proofs is a description of the {\it completion} of a subspace $L_\pi$
of matrix elements of a representation $\pi$ (Theorem \ref{th-13} and Theorem \ref{th-K-mat}). 
This is a direct corollary of the description of Grothendieck semigroup of subspaces of matrix elements 
(Proposition \ref{prop-1}). 
%The results from the representation theory are recalled whenever they are needed.

The Ku\v{s}nirenko theorem is a particular case of the Bernstein-Kushnirenko theorem
when all the Newton polytopes of equations are the same. Today,
several generalizations of the Ku\v{s}nirenko theorem are known (see \cite{Brion} for spherical
varieties, and \cite{Askold-Kiumars-Newton-Okounkov} for arbitrary varieties).
On the other hand, the Bernstein--Kushnirenko theorem is
harder to generalize (see \cite{Askold-Kiumars-reductive}). Its most important generalization
is the Kazarnovskii theorem which we discuss in this note.

\section{Intersection theory of finite dimensional subspaces of regular functions} \label{sec-int-index}
Let $X$ be a complex $n$-dimensional irreducible normal affine variety with $\c[X]$ its ring of regular
functions. Consider the collection $K[X]$ of all nonzero finite dimensional subspaces of
$\c[X]$. The {\it product} of two subspaces $L_1, L_2 \in K[X]$ is the subspace spanned by all the
$fg$ where $f \in L_1$, $g \in L_2$. With this product $K[X]$ is a commutative semigroup.

The {\it base locus} $Z(L)$ of a subspace $L \in K[X]$ is the set of all $x \in X$
for which $f(x) = 0$ for any $f \in L$. Let $L_1, \ldots, L_n \in K[X]$ and $Z = \bigcup_i Z(L_i)$.
\begin{Def}
The {\it intersection index} $[L_1, \ldots, L_n]$ is the number of solutions in $X \setminus Z$
of a generic system of equations $f_1 = \cdots = f_n = 0$ where $f_i \in L_i$, $1 \leq i \leq n$.
\end{Def}

One shows that the intersection index is well-defined (i.e. is independent of the choice of a
generic system) \cite{Askold-Kiumars-MMJ}. It is obvious that the intersection index is symmetric with respect
to permuting the subspaces $L_i$. Moreover, the intersection index is linear in each argument.
The linearity in first argument means:
\begin{equation} \label{equ-*}
[L'_1L''_1, L_2,\dots, L_n]= [L'_1, L_2,\dots, L_n] + [L''_1, L_2,\dots,L_n],
\end{equation}
for any $L'_1, L''_1, L_2, \ldots, L_n \in K[X]$.
From (\ref{equ-*}) one sees that for a fixed $(n-2)$-tuple of subspaces
$L_2, \ldots, L_n \in K[X]$, the map $\pi: K[X] \to \r$ given by $\pi(L) = [L, L_2, \ldots, L_n]$
is a homomorphism from the semigroup $K[X]$ to the additive group of integers. The existence of such
a homomorphism shows that the intersection index induces an intersection index on
$\Gr(K[X])$, i.e. {\it the intersection index $[L_1, \ldots, L_n]$ remains invariant if we substitute each
$L_i$ with an analogous subspace $\tilde{L}_i$.}

One can describe the relation of analogous subspaces in a different way as follows (see \cite{Askold-Kiumars-MMJ}).
A rational function $f \in \c(X)$ is called {\it integral over the subspace $L$} if
it satisfies an equation $$f^m+a_1 f^{m-1} + \dots a_0 =0$$ with $m>0$ and $a_i \in L^i$, $1 \leq i \leq m$.
The collection of all the rational functions integral over $L$ forms a finite dimensional subspace
$\overline{L}$ called the {\it completion of $L$}.

\begin{Prop} \label{prop-1}
1) Two subspaces $L_1, L_2 \in K[X]$ are analogous if and only if $\overline{L}_1 = \overline{L}_2$.
2) For any $L \in K[X]$, the completion $\overline{L}$ belongs to $K[X]$ and
is analogous to $L$.
3) Moreover, the completion $\overline{L}$ contains all the subspaces $M \in K[X]$ analogous to $L$.
\end{Prop}

For $L \in K[X]$ define the Hilbert function $H_L$ by $H_L(k) = \dim (\overline{L^k})$.
The following theorem provides a way to compute the self-intersection index of a subspace $L$
(see \cite[Part II]{Askold-Kiumars-Newton-Okounkov}):
\begin{Th} \label{th-intersec-index-Hilbert-function}
For any $L \in K[X]$, the limit $$a(L) = \lim_{k \to \infty} H_L(k)/k^n$$ exists, and
the self-intersection index $[L, \ldots, L]$ is equal to $n! a(L)$.
\end{Th}
The proof is based on the Hilbert theorem on the dimension and degree of a subvariety of the projective space.

\section{Mixed volume and mixed integral} \label{sec-mixed-vol}
A function $F: \L \to \r$ on a (possibly infinite dimensional) linear space $\L$ is called
a homogeneous polynomial of degree $k$ if its restriction to any finite dimensional subspace of $\L$
is a homogeneous polynomial of degree $k$. (For any $k$, the zero constant function is a homogeneous
polynomial of degree $k$.)

\begin{Def}
To a symmetric multi-linear function $B(v_1, \ldots, v_k)$, $v_i \in \L$ one corresponds a homogeneous polynomial of
degree $k$ on $\L$ by $P(v) = B(v, \ldots, v)$. We say that symmetric form $B$ is a
{\it polarization of the homogeneous polynomial $P$}.
\end{Def}

If $F$ is a homogeneous polynomial of degree $k$, then its derivative $F'_v(x)$ in the direction of a
vector $v$ is linear in $v$ and homogeneous of degree $k-1$ in $x$. Let $v_1, \ldots, v_k$ be
an $k$-tuple of vectors. With each $x$, the $k$-th derivative $F^{(k)}_{v_1, \ldots, v_k}(x)$
is a symmetric multi-linear function in the $v_i$. One easily verifies the following:

\begin{Prop}
Any homogeneous polynomial of degree $k$ has a unique {\it polarization} $B$ defined by the
formula: $$B(v_1,\dots,v_k)= (1/k!) F^{(k)}_{v_1,\dots,v_k}.$$
\end{Prop}

A compact convex subset of $\r^n$ is called a {\it convex body}.
Consider the collection of convex bodies in $\r^n$. We have two
operations of Minkowski sum and multiplication by a non-negative
scalar for the convex bodies. The collection of convex bodies with
Minkowski addition is a semigroup with cancelation. The
multiplication by a non-negative scalar is associative and
distributive with respect to the Minkowski addition. These
properties allow us to extend the collection of convex bodies to the
(infinite dimensional) linear space $\L$ of {\it virtual convex
bodies} consisting of formal differences of convex bodies (see
\cite{Burago-Zalgaller}).

Let $d\mu = dx_1 \cdots dx_n$ be the standard  Euclidean measure in $\r^n$.
For each convex body $\Delta \subset \r^n$ let $\Vol(\Delta) = \int_\Delta d\mu$ be its volume.
The following statement is well-known:

\begin{Prop}
The function $\Vol$ has a unique extension to the linear space $\L$ of virtual convex bodies as a
homogeneous polynomial of degree $n$.
\end{Prop}

\begin{Def}
The {\it mixed volume} $V(\Delta_1, \ldots, \Delta_n)$ of the convex bodies $\Delta_i$
is the value of the polarization of the volume polynomial $\Vol$ at $(\Delta_1, \ldots, \Delta_n)$.
\end{Def}

Fix a homogeneous polynomial $F$ of degree $p$ in $\r^n$. Let
$IF(\Delta) = \int_\Delta F d\mu$ denote the integral of $F$ on $\Delta$.
One has the following (see for example \cite{Kh-P}):

\begin{Prop}
The function $IF$ has a unique extension to the linear space $\L$ of virtual convex bodies as
a homogeneous polynomial of degree $n+p$.
\end{Prop}

\begin{Def}
The mixed integral $IF(\Delta_1, \ldots, \Delta_{n+p})$ of a homogeneous polynomial $F$ over the bodies
$\Delta_1, \ldots, \Delta_{n+p}$ is the value of the polarization of the
polynomial $IF$ at the bodies $\Delta_1, \ldots, \Delta_{n+p}$.
\end{Def}
From definition, the mixed integral of the constant polynomial $F \equiv 1$ is the mixed volume.

\section{Theorems of Bernstein-Ku\v{s}nirenko and Kazarnovskii} \label{sec-Bern-Kush}
We start with the Bernstein-Ku\v{s}nirenko theorem on the number of solutions of a system of Laurent
polynomials. The characters $x^k = x_1^{k_1} \cdots x_n^{k_n}$ of the (multiplicative)
group $(\c^*)^n$ are in one-to-one correspondence with the points $k = (k_1, \ldots, k_n)$ in the lattice $\z^n$.
A {\it Laurent polynomial} is a regular function in $(\c^*)^n$. Any Laurent polynomial $f$  is a
linear combination of the characters, namely $f = \sum_k c_k x^k$. The {\it support} of a Laurent polynomial $f$
is the set of $k \in \z^n$ with $c_k \neq 0$. For each nonempty set $A \subset \z^n$ one defines
the linear subspace $L_A$ to be collection of all the Laurent polynomials whose supports are contained in $A$.\\

\noindent {\bf Problem:} Given an $n$-tuple of finite nonempty subsets $A_1, \ldots, A_n \subset \z^n$
find the intersection index $[L_{A_1}, \ldots, L_{A_n}]$ in $(\c^*)^n$.

The Bernstein-Ku\v{s}nirenko theorem answers this problem completely \linebreak 
(\cite{Kushnirenko} and \cite{Bernstein}).

\begin{Def}
The {\it Newton polytope} $\Delta(A)$ of a nonempty finite subset $A \in \z^n$ is the
convex hull of $A$.
\end{Def}

\begin{Th}[Bernstein-Ku\v{s}nirenko]
The intersection index $[L_{A_1}, \ldots, L_{A_n}]$ is equal to $$n! V(\Delta(A_1), \ldots, \Delta(A_n)).$$
\end{Th}

Next we discuss the Kazarnovskii theorem. Let $G$ be a complex
connected $m$-dimensional reductive algebraic with a maximal torus
$T \cong (\c^*)^n$. We identify the lattice of characters of $T$ with
$\z^n$ and its real  pan with $\r^n$. The Weyl group $W$ of $G$
(which is a finite group of reflections) acts on $\r^n$ and maps
$\z^n$ to itself. One fixes a polyhedral cone $C$ which is a
fundamental domain for the action of $W$ on $\r^n$ and call it the
{\it positive Weyl chamber}. It is the main result of the highest
weight theory that the finite dimensional irreducible
representations of $G$ are in one-to-one correspondence with the
integral points in the positive Weyl chamber. An integral point in
the positive Weyl chamber is called a {\it dominant weight}. For a
dominant weight $\lambda$ we denote its corresponding finite
dimensional irreducible representation by $V_\lambda$. The point
$\lambda$ is called the {\it highest weight} of the representation
$V_\lambda$.

Consider the left action of the group $G$ on itself. The induced action on the ring of regular
functions $\c[G]$ is given by $g \cdot f(h) = f(g^{-1}h)$ where $g \in G$ and $f \in \c[G]$.
To a function $f \in \c[G]$ we associate the subspace
$L_f$ which is the smallest $G$-invariant subspace of $\c[G]$ containing $f$. One can show that
$L_f$ is finite dimensional (i.e. $\c[G]$ is a so-called {\it rational $G$-module}).

\begin{Def}
The {\it spectrum} $\Spec(f)$ of a function $f \in \c[G]$ is the set of all integral points $\lambda$ for which
$V_\lambda$ appears in the decomposition of $L_f$ into a direct sum of irreducible representations.
\end{Def}

\begin{Def}
For a nonempty subset $A \in C \cap \z^n$ let $L_A$ be the finite dimensional $G$-invariant
subspace of $\c[G]$ consisting of all $f$ with $\Spec(f) \subset A$.
\end{Def}

\noindent {\bf Problem:} Given an $m$-tuple of finite nonempty subsets of
dominant weights $A_1, \ldots, A_m \subset C \cap \z^n$,
find the intersection index $[L_{A_1}, \ldots, L_{A_m}]$ in $G$.

The Kazarnovskii theorem answers this completely (\cite{Kazarnovskii}).
We will need an extra  notation.
According to the {\it Weyl dimension formula} the dimension of an irreducible representation
$V_\lambda$ is equal to $F_W(\lambda)$,
where $F_W$ is a  polynomial on $\r^n$ of degree $(m-n)/2$ defined explicitly in terms of data associated to $W$.
We call $F_W$ the {\it Weyl polynomial of $W$}. We denote the homogeneous component of highest degree of $F_W$ by
$\phi_W$.

\begin{Def}[Weight polytope]
Let $A \subset C \cap \z^n$ be a finite nonempty subset of dominant weights.
The {\it weight polytope $\Delta_W(A)$} is the convex hull of the union of
Weyl orbits of elements of $A$.
\end{Def}

\begin{Th}[Kazarnovskii] \label{th-Kazarnovskii}
Given the finite nonempty subsets $A_1, \ldots, A_m \in C \cap \z^n$
the intersection index
$[L_{A_1}, \ldots, L_{A_m}]$ is equal to the mixed integral $$(m!/\#W) I\phi_W^2(\Delta_W(A_1), \ldots,
 \Delta_W(A_m)),$$
where $\#W$ is the number of elements in the Weyl group $W$.
\end{Th}

\begin{Def}[Moment polytope]
Let $A$ be a finite nonempty subset of $C \cap \z^n$.
The {\it moment polytope} $\Delta^+_W(A)$ is the intersection of $\Delta_W(A)$ with the positive Weyl chamber.
\end{Def}

\begin{Rem}
\noindent 1) Note that contrary to the weight polytope, the moment polytope $\Delta^+_W(A)$ is not
necessarily an integral polytope.\\

\noindent 2) The polytope $\Delta^+_W(A)$ can be identified with the moment (or Kirwan) polytope of
a certain compactification of the group $G$ as a $K \times K$-Hamiltonian space, where $K$ is a
maximal compact subgroup of $G$. This justifies the term moment polytope.
\end{Rem}

\begin{Cor}[Alternative statement of the Kazarnovskii theorem]
Given the finite nonempty subsets $A_1, \ldots, A_m \in C \cap \z^n$
the intersection index  $[L_{A_1}, \ldots, L_{A_m}]$ is equal to the mixed integral
$$m! I\phi_W^2(\Delta^+_W(A_1), \ldots, \Delta^+_W(A_m)).$$
\end{Cor}
\begin{proof}
The corollary follows from the Kazarnovskii theorem (Theorem \ref{th-Kazarnovskii}) and the invariance of
polynomial $\phi_W^2$ under the Weyl group action.
\end{proof}

\begin{Rem}
\noindent 1) The Bernstein-Ku\v{s}nirenko theorem is a particular case of the Kazarnovskii theorem for $G = (\c^*)^n$.
For $G = (\c^*)^n$, the Weyl group is trivial and the positive Weyl chamber is the whole $\r^n$. For each
$\lambda \in \z^n$, the representation $V_\lambda$ is the $1$-dimensional space on which $(\c^*)^n$ acts by
multiplication via the character $\lambda$.\\

\noindent 2) Any subspace $L_A$ is $G$-invariant, so its base locus $Z(L_A)$ is also $G$-invariant and hence is
empty or the whole $G$. But since $L_A$ is nonzero its base locus should be empty.
Thus the Kazarnovskii theorem computes the
number of solutions of a generic system of equations $f_1 = \cdots f_m = 0$ where $f_i \in L_{A_i}$,
in the whole group $G$ (rather than $G \setminus Z$).\\

\noindent 3) For a classical group $G$ the formula in
Kazarnovskii theorem can be rewritten in terms of the mixed volume
of certain polytopes (see Section \ref{sec-G-C}).
\end{Rem}

\section{Semigroup of finite sets with respect to addition} \label{sec-finite-sets}
There is an addition operation on the collection of subsets of $\r^n$. The sum of two sets $A$ and $B$ is the set
$A+B = \{a+b \mid a \in A,~ b \in B\}$. One verifies that the sum of two convex bodies (respectively convex
integral polytopes) is again a convex body (respectively a convex integral polytope).
This is the well-known Minkowski sum of convex bodies. Consider the following:
\begin{itemize}
%\item[-] $K$ the set of finite subsets in $\z^n$.
\item[-] $\K$, the semigroup of all finite subsets of $\z^n$ with the addition of subset.
%\item[-] $P$ the set of convex integral polytopes in $\r^n$.
\item[-] $\P$, the semigroup of all convex integral polytopes with the Minkowski sum.
\end{itemize}

\begin{Prop} \label{prop-7}
The semigroup $\P$ has cancelation property.
\end{Prop}
Proposition \ref{prop-7} follows from the more general fact that the semigroup of convex bodies with respect to
the Minkowski sum has cancelation property.
The next statement is easy to verify:
\begin{Prop} \label{prop-8}
The map which associates to a finite nonempty set $A \subset \z^n$ its convex hull $\Delta(A)$,
is a homomorphism of semigroups from $\K$ to $\P$.
\end{Prop}

For an integral convex polytope $\Delta \in \P$ let $\Delta_\z \in \K$ denote the finite set
of integral points in $\Delta$, i.e. $\Delta_\z = \Delta \cap \z^n$.
It is not hard to verify the following (see \cite{Askold-sum-of-finite-sets}):
\begin{Prop} \label{prop-9}
For any nonempty subset $A \subset \z^n$ we have:
$$A + n\Delta(A)_\z = (n+1)\Delta(A)_\z = \Delta(A)_\z + n\Delta(A)_\z.$$
\end{Prop}

We then have the following description for the Grothendieck semigroup of $\K$.
\begin{Th} \label{th-10}
The Grothendieck semigroup of $\K$ is isomorphic to $\P$. The homomorphism $\rho: \K \to \P$ is
given by $\rho(A) = \Delta(A)$.
\end{Th}
\begin{proof}
From Proposition \ref{prop-8} it follows that if $A \sim B$ then $\Delta(A) \sim \Delta(B)$.
Conversely, from Proposition \ref{prop-9} we know that $A$ and $\Delta(A)_\z$ are analogous.
By definition if $\Delta(A) = \Delta(B)$ then $\Delta(A)_\z = \Delta(B)_\z$. So if $\Delta(A)
= \Delta(B)$ then $A$ and $B$ are analogous.
\end{proof}

\section{Subspaces of matrix elements for $(\c^*)^n$}  \label{sec-matrix-elements-torus}
In the ring of regular functions on $(\c^*)^n$  there is a collection of subspaces which is of
particular interest, namely the subspaces of the form $L_A$ where $A \in \z^n$ is a nonempty finite subset.
The following are obvious:
\begin{Prop} \label{prop-11}
1) For any finite nonempty set $A \subset \z^n$ the dimension of the subspace $L_A$ is equal to the
number of points in $A$.
2) For finite nonempty sets $A, B \subset \z^n$ we have $L_AL_B = L_{A+B}$.
\end{Prop}

A finite nonempty subset $A \subset \z^n$ gives a finite dimensional diagonal representation
$\pi_A: (\c^*)^n \to (\c^*)^r \subset \GL(r, \c)$ where $r = \#A$. The subspace $L_A$ is in fact the space of
matrix elements \footnote{Any linear combination of the matrix entries of a representation,
regarded as a function on the group, is called a {\it matrix element} of the given representation.}
of the representation $\pi_A$.

\begin{Def}
According to Proposition \ref{prop-11}(2), the collection of subspaces $L_A$ is a semigroup
with respect to product of subspaces. We call it the
{\it semigroup of subspaces of matrix elements of $(\c^*)^n$} and denote it by $K_{\textup{mat}}[(\c^*)^n]$.
\end{Def}

\begin{Th} \label{th-13}
The Grothendieck semigroup of $K_{\textup{mat}}[(\c^*)^n]$ is isomorphic to $\P$. The homomorphism
$\rho: K_{\textup{mat}}[(\c^*)^n] \to \P$ is given by $\rho(L_A) = \Delta(A)$.
The completion $\overline{L_A}$ of $L_A$ is equal to  $L_B$ where $L_B= \Delta_\z(A)$.
\end{Th}
\begin{proof}
From Proposition \ref{prop-11} we know that the semigroup
$K_{\textup{mat}}[(\c^*)^n]$ is isomorphic to $\K$. Also by Theorem \ref{th-10}
the Grothendieck semigroup of $\K$ is $\P$.
The equality $\overline{L_A} =L_B$  now follows from Proposition \ref{prop-1}.
\end{proof}

\begin{proof}[Proof of the Bernstein-Ku\v{s}nirenko theorem]
First let us prove the Ku\v{s}nirenko theorem, namely {\it for any
nonempty finite subset $A \subset \z^n$, the self-intersection index
$[L_A, \ldots, L_A]$ is equal to $n! \Vol(\Delta(A))$.} According to
Theorem \ref{th-13} we have $\overline{L_A^k} = L_{B_k}$ where $B_k
= (k\Delta(A))_\z$ is the set of integral points in $k\Delta(A)$. By
Proposition \ref{prop-11} the dimension of $L_{B_k}$ is equal to
$\#B_k$, i.e. the number of integral points in the polytope
$k\Delta(A)$. Put $H(k) = \dim(L_{B_k}) = \#B_k$. Note that, as $k
\to \infty$, the number of integral points in $k\Delta(A)$
asymptotically equals to the volume of the polytope $k\Delta(A)$. It
follows that  the limit $\lim_{k \to \infty} H(k)/k^n$ is equal to
$\Vol(\Delta(A))$. (See \cite[Part
I]{Askold-Kiumars-Newton-Okounkov} for a more detailed study of
asymptotic of such kind.) We can now conclude the Ku\v{s}nirenko
theorem as a corollary of Theorem
\ref{th-intersec-index-Hilbert-function}. The
Bernstein-Ku\v{s}nirenko theorem automatically follows from the
Ku\v{s}nirenko theorem and the multi-linearity of intersection index
and mixed volume.
\end{proof}

\section{Semigroup of representations up to spectral equivalence} \label{sec-main-spectrally-equiv}
In this section we describe the Grothendieck semigroup of the
semigroup of finite dimensional representations with tensor product and up to the spectral
equivalence.

We will need a generalization of Theorem \ref{th-10}. Let $\K_0 \subset \K$ be a subset in
$\K$ equipped with some addition operation $\tilde{+}$ with respect to which $\K_0$ is a
semigroup. Assume that $(\K_0, \tilde{+})$ satisfies the following properties:

\begin{enumerate}
\item If $A \in \K_0$ then $\Delta(A)_\z \in \K_0$.
\item  If $A, B \in \K_0$ then $A+B \subset A \tilde{+} B$.
\item If $A, B \in \K_0$ then $\Delta(A \tilde{+} B) = \Delta(A+B)$.
\end{enumerate}

With the semigroup $(\K_0, \tilde{+})$ let us associate the semigroup
$\P_0 \subset \P$ whose elements are the integral polytopes of the form
$\Delta(A)$ for $A \in \K_0$ and the addition operation in $\P_0$ is the
Minkowski sum (by the property (3) the set $\P_0$ is closed under the
Minkowski sum).

Repeating word by word the proof of Theorem \ref{th-10}, with $(\K_0, \tilde{+})$
instead of $\K$ and $\P_0$ instead of $\P$, we obtain:
\begin{Th} \label{th-14}
The Grothendieck semigroup of $(\K_0, \tilde{+})$ is isomorphic to $\P_0$. The
homomorphism $\rho: (\K_0, \tilde{+}) \to \P_0$ is given by $\rho(A) = \Delta(A)$.
\end{Th}

Now let $G$ be a complex connected reductive algebraic group.
\begin{Def}
The {\it spectrum} $\Spec(\pi)$ of a  finite dimensional representation $\pi$ of $G$
is the set of all dominant weights $\lambda$
where $V_\lambda$ appears in the decomposition of $\pi$ as a direct sum of irreducible representations.
\end{Def}

We say that two finite dimensional representations $\pi_1, \pi_2$ are
{\it spectrally equivalent} if their spectrums coincide (multiplicities of the irreducible
representations appearing in $\pi_1$ and $\pi_2$ can be different).
It is easy to see that the tensor product of representations respects the spectral equivalence.

\begin{Def}
We denote the semigroup of all the finite dimensional
representations of $G$ up to the spectral equivalence, and with respect to the
tensor product of representations, by $\R_{\Spec}(G)$.
\end{Def}

Let us introduce the following semigroups:
\begin{itemize}
\item [-] $\K_W$, the subsemigroup of $\K$ consisting of finite subsets which are invariant under $W$.
\item [-] $\P_W$, the subsemigroup of $\P$ consisting of convex integral polytopes which are invariant under $W$.
\end{itemize}

\begin{Def}
For a representation $\pi$ let $\Spec_W(\pi) \in K_W$ be the
union of all the Weyl orbits of elements of $\Spec(\pi)$, i.e.
$\Spec_W(\pi) = \{ w(\lambda) \mid \lambda \in \Spec(\pi), ~ w \in W\}$.
We call the convex hull of the set $\Spec_W(\pi)$ the {\it weight polytope} of $\pi$ and
denote it by $\Delta_W(\pi)$. Also the intersection of the weight polytope
with the positive Weyl chamber will be called the {\it moment polytope} of $\pi$
and denoted by $\Delta_W^+(\pi)$.
\end{Def}

Consider the map $\Spec_W: \R_{\Spec}(G) \to \K_W$ which associates to a representation $\pi$
the set $\Spec_W(\pi)$.

\begin{Prop} \label{prop-17}
1) The map $\Spec_W$ is onto, i.e.
for any subset $A \in \K_W$ there is a representation $\pi$ with $\Spec_W(\pi)= A$.
2) For any polytope $\Delta\in P_W$ there is a representation $\pi$ with $\Delta_W(\pi) = \Delta$.
\end{Prop}
\begin{proof}
1) Let $A_0$ be the intersection of $A$ with the positive Weyl chamber. Let $\pi$ be the
direct sum of the irreducible representations $V_\lambda$ for $\lambda \in A_0$. Then
$\Spec_W(\pi) = A$. 2) follows from 1).
\end{proof}

The tensor product of representations induces a binary operation on
the set $\K_W$ which we denote it by $\tilde{+}$.
From definition and Proposition \ref{prop-17}(1) the map $\Spec_W: (\R_{\Spec}(G), \otimes) \to
(\K_W, \tilde{+})$ is an isomorphism of semigroups.

For a representation $\pi$ of $G$, let $\chi(\pi) \subset \z^n$ be the set of characters of the restriction of
$\pi$ to the torus $T$.
It is well-known that the set $\chi(\pi)$ is invariant under the Weyl group $W$. From the representation theory
of torus $T$ we have:
\begin{Prop} \label{prop-15}
For any two representations $\pi_1, \pi_2$ one has $\chi(\pi_1 \otimes \pi_2) = \chi(\pi_1) + \chi(\pi_2)$.
\end{Prop}

And from the highest weight theory for the group $G$ one obtains:
\begin{Prop} \label{prop-16}
The convex hull of $\chi(\pi)$ coincides with $\Delta_W(\pi)$.
\end{Prop}

We will need the following key fact regarding tensor product of finite dimensional representations.
It is commonly known as the PRV conjecture. It was first conjectured by
K. Parthasarathy, R. Ranga Rao and V. Varadarajan in \cite{PRV}.
Later it was proved by S. Kumar in \cite{Kumar}.
\begin{Th} \label{th-18}
For any two finite dimensional representations $\pi_1, \pi_2$ of $G$ we have:
$\Spec_W(\pi_1) + \Spec_W(\pi_2)\subseteq
\Spec_W(\pi_1 \otimes \pi_2).$
\end{Th}

From Proposition \ref{prop-15} and Proposition \ref{prop-16} one readily has the following
property of the weight polytope:
\begin{Prop} \label{prop-19}
For any two finite dimensional representations $\pi_1, \pi_2$ of $G$ we have:
$\Delta_W (\pi_1) + \Delta_W(\pi_2)= \Delta_W(\pi_1 \otimes \pi_2).$
\end{Prop}

The following is the main result of the paper which describes
the Grothendieck semigroup of $(\R_\Spec(G), \otimes)$.
\begin{Th}[Main theorem] \label{th-20}
The Grothendieck semigroup for $\R_{\Spec}(G)$ is isomorphic to $\P_W$.
The homomorphism $\rho: \R_{\Spec}(G) \to \P_W$  maps a representation $\pi$
to its weight polytope $\Delta_W(\pi)$.
\end{Th}
\begin{proof}
It is enough to show that the semigroup $(\K_W,\tilde{+})$ is
isomorphic to the semigroup  $\P^W$.
According to Propositions \ref{prop-17} and \ref{prop-19}
and Theorem \ref{th-18} the semigroup $(\K_W, \tilde{+})$
satisfies the properties (1)-(3) stated before Theorem \ref{th-14}.
The theorem now follows from Theorem \ref{th-14}.
\end{proof}

Let us give another formulation of Theorem \ref{th-20}. Denote by $\P^+_W$ the semigroup of all polytopes
which can be represented as $\Delta \cap C$ for $\Delta \in \P_W$ together with the
Minkowski sum. It is easy to see that the map $\pi: \P_W \to \P^+_W$ defined by
$\pi(\Delta) = \Delta \cap C$ is an isomorphism of semigroups.

\begin{Th} \label{th-21}
The Grothendieck semigroup of $\R_{\Spec}(G)$ is isomorphic to $\P^+_W$.
The homomorphism $\rho: \R_{\Spec}(G) \to \P^+_W$  maps a representation $\pi$
to the moment polytope $\Delta^+_W(\pi)$.
\end{Th}

\section{Subspaces of matrix elements for reductive groups} \label{sec-matrix-elements-reductive}
The subspaces $L_A$ appearing in the Kazarnovskii theorem
can be realized in an alternative way: they are in fact the subspaces of matrix elements
of representations of $G$. Let $\pi$ be a finite dimensional representation.
Denote by $L_\pi \subset \c[G]$ the linear subspace
spanned by the matrix elements of the representation $\pi$.
It is invariant under the left action (as well as the right action)
of $G$ on $\c[G]$. We will regard it as a $G$-submodule of $\c[G]$ (for the left action).
\begin{Prop} \label{prop-22}
One has $L_\pi = L_A$ where $A = \Spec(\pi)$.
\end{Prop}

\begin{Cor} \label{cor-23}
If two representations $\pi_1$ and $\pi_2$ are spectrally equivalent then $L_{\pi_1} = L_{\pi_2}$.
\end{Cor}

It is well-known that every irreducible representation $V_\lambda$ appears in the
(left) regular representation of $G$ on $\c[G]$ with multiplicity equal to $\dim(V_\lambda)$.
From this we get the following.
\begin{Prop} \label{prop-24}
Let $L_\pi = \sum_\lambda m_\lambda V_\lambda$ be a decomposition of $L_\pi$ into
irreducible representations. Put $A = \Spec(\pi)$. Then: 1) if $\lambda \in A$ we have $m_\lambda = \dim (V_\lambda) = F_W(\lambda)$;
2) if $\lambda \notin A$ then $m_\lambda = 0$.
\end{Prop}

From Proposition \ref{prop-24} and the Weyl dimension formula we obtain:
\begin{Cor} \label{cor-25}
For any finite nonempty subset $A \subset C \cap \z^n$, $\dim(L_A) = \sum_{\lambda \in A} F_W^2(\lambda)$.
\end{Cor}

The following is straight forward to verify.
\begin{Prop} \label{prop-26}
For any two representations $\pi_1$, $\pi_2$ one has
$L_{\pi_1 \otimes \pi_2} = L_{\pi_1} L_{\pi_2}$.
\end{Prop}

\begin{Def}
According to Proposition \ref{prop-26} the subspaces $L_A$ form a semigroup (with respect to
product of subspaces). We call this semigroup
the {\it semigroup of matrix elements of $G$} and denote it by $K_{\textup{mat}}[G]$.
\end{Def}

\begin{Th} \label{th-K-mat}
The semigroup $K_{\textup{mat}}[G]$ is isomorphic to $\R_{\Spec}(G)$ and its Grothendieck semigroup
is isomorphic to $\P^+_W$. The map $\rho: K_{\textup{mat}}[G] \to \P^+_W$ is given by
$\rho(L_A) = \Delta^+_W(A)$. The completion of $L_A$ is $L_B$ where $B=\Delta^+_W(A)_\z$.
\end{Th}
\begin{proof}
According to Corollary \ref{cor-23} and Propositions \ref{prop-24} and \ref{prop-26}, the
map $\pi \mapsto L_\pi$ is an isomorphism of semigroups $\R_{\Spec}(G)$ and
$K_{\textup{mat}}[G]$. Then by Theorem \ref{th-21} the
Grothendieck semigroup of $K_{\textup{mat}}[G]$ is isomorphic to $\P^+_W$. The equality $\overline{L_A}
= L_B$ follows from Proposition \ref{prop-1}.
\end{proof}

We can now prove the Kazarnovskii theorem.
As in the Bernstein-Ku\v{s}nirenko theorem, first we prove it for the self-intersection index.
\begin{Lem}[Analogue of the Ku\v{s}nirenko theorem for a group $G$] \label{lem-28}
For any finite nonempty set $A \subset C \cap \z^n$,
the self-intersection index $[L_A, \ldots, L_A]$ is equal to $m! I\phi^2_W(\Delta^+_W(A))$.
\end{Lem}
\begin{proof}
According to Theorem \ref{th-K-mat} we have $\overline{L_A^k} = L_{B_k}$ where $B_k = (k\Delta^+_W(A))_\z$. By
Corollary \ref{cor-25} the dimension of $L_{B_k}$
is equal to $\sum_{\lambda \in B_k} F^2_W(\lambda)$.
Put $H(k) = \dim(L_{B_k})$. One sees that, as $k \to \infty$,
the sum $\sum_{\lambda \in B_k} F^2_W(\lambda)$ asymptotically is equal to
$k^m \int_{\Delta^+_W(A)} \phi^2_W d\mu$, because the polynomial $F^2_W$ has
degree $m-n$ and its homogeneous component of highest degree is $\phi^2_W$.
It follows that
$\lim_{k \to \infty} H(k)/k^m$ is equal to $\int_{\Delta^+_W(A)} \phi^2_W d\mu$.
(See \cite[Part I]{Askold-Kiumars-Newton-Okounkov} for a more detailed study of asymptotics of
such kind.) We can now conclude the lemma from Theorem \ref{th-intersec-index-Hilbert-function}.
\end{proof}

\begin{proof}[Proof of the Kazarnovskii theorem]
The Kazarnovskii theorem follows from Lemma \ref{lem-28}
and the multi-linearity of the mixed integral of a homogeneous polynomial as well as the multi-linearity
of the intersection index.
\end{proof}

\section{Intersection index as mixed volume} \label{sec-G-C}
In this section we see how to rewrite the formula in the Kazarnovskii theorem as
a mixed volume of certain polytopes (instead of mixed integral). To this end, we
use the so-called {\it Gelfand-Cetlin polytopes}.

In their classical paper \cite{G-C}, Gelfand and Cetlin constructed
a natural basis for any irreducible representation of $\GL(n, \c)$
and showed how to parameterize the elements of this basis with
integral points in a certain convex polytope. These polytopes are
called the {\it Gelfand-Cetlin polytopes}. Since then similar
constructions have been done for other classical groups and
analogous polytopes were defined (see \cite{B-Z1}). We will also
call them  {\it Gelfand-Cetlin polytopes} or for short {\it G-C
polytopes}. Consider the list of groups $\c^*$, $\SL(n_1, \c)$, $\SO(n_2,
\c)$ and $\SP(2n_3, \c)$, for any $n_1, n_2, n_3 \in \n$.
We will say that $G$ is a {\it classical group}
if $G$ is in this list, or if $G$ can be constructed from the
groups in the list using the operations of taking direct product
and/or taking quotient by a finite central subgroup.
In this sense, the general linear group and the
orthogonal group are classical groups.

Let $G$ be a classical group. As usual let $m = \dim(G)$, and we
identify the weight lattice of $G$ with $\z^n$, its real span by
$\r^n$, and denote the positive Weil chamber by $C$. In summary
we have:
\begin{Th}[G-C polytopes] \label{th-G-C-additive}
For any classical group $G$ and for any $\lambda \in C$ one
can explicitly construct a polytope $\Delta_{GC}(\lambda) \subset
\r^{(m-n)/2}$, called the Gelfand-Cetlin polytope of $\lambda$, with
the following properties:
\begin{enumerate}
\item If $\lambda$ is integral then the dimension of $V_\lambda$ is equal to the number of
integral points in $\Delta_{GC}(\lambda)$.
\item The map $\lambda \mapsto \Delta_{GC}(\lambda)$ is linear, i.e.
for any two $\lambda, \gamma \in C$ and $c_1, c_2 \geq 0$ we
have $= \Delta_{GC}(c_1\lambda + c_2\gamma) =
c_1\Delta_{GC}(\lambda) + c_2\Delta_{GC}(\gamma).$
\end{enumerate}
\end{Th}
The part (2) in the above theorem is an immediate corollary of the defining inequalities of
the G-C polytopes for the classical groups.

\begin{Def}
Let $A$ be a finite nonempty set of dominant weights  of $G$.
Define the polytope $\tilde{\Delta}(A) \subset C \times \r^{(m-n)}$
by:
$$\tilde{\Delta}(A) = \bigcup_{\lambda \in \Delta^+_W(A)} \{(\lambda, x, y) \mid x,y \in \Delta_{GC}(\lambda) \}.$$
In other words, the projection on the first factor maps
$\tilde{\Delta}(A)$ to the weight polytope $\Delta^+_W(A)$ and the
fibre over each $\lambda$ is the double G-C polytope
$\Delta_{GC}(\lambda) \times \Delta_{GC}(\lambda)$.
\end{Def}

\begin{Th}[Reformulation of the Kazarnovskii theorem]
Given finite nonempty subsets $A_1, \ldots, A_m \subset C \cap \z^n$, the
intersection index  $[L_{A_1}, \ldots, L_{A_m}]$ is equal to the
mixed volume $V(\tilde{\Delta}^+_W(A_1), \ldots,
\tilde{\Delta}^+_W(A_m))$ multiplied with $m!$.
\end{Th}
\begin{proof}
%its $(m-n)/2$-dimensional volume is a
%polynomial of degree $(m-n)/2$ in $\lambda$.
Because the polytope $\Delta_{GC}(\lambda)$ depends linearly on
$\lambda$, for a nonempty finite subset $A$ of dominant weights the map
$A \mapsto \tilde{\Delta}(\pi_A)$ is a linear map with respect to the
addition of subsets.
Hence to prove the theorem it is enough to verify that the self-intersection index of the
subspace $L_A$ is equal to the volume $\Vol(\tilde{\Delta}^+_W(A))$ multiplied with $m!$.
The number of integral points in $k\Delta_{GC}(\lambda) = \Delta_{GC}(k\lambda)$,
for large $k$, is asymptotically equal to the volume of $k\Delta_{GC}(\lambda)$. Now using
Theorem \ref{th-G-C-additive}(1) and the Weil dimension
formula we have $\Vol_{(m-n)/2}(\Delta_{GC}(\lambda)) = \phi_W(\lambda)$.
So the $(m-n)$-dimensional volume of $\Delta_{GC}(\lambda) \times
\Delta_{GC}(\lambda)$ is equal to $\phi_W^2(\lambda)$. The theorem
then follows from Lemma \ref{lem-28} and the Fubini theorem.
\end{proof}

\begin{Rem}
\noindent 1) The construction of $\tilde{\Delta}(\pi)$ goes back to A. Okounkov (\cite{Okounkov-spherical})
who introduced such polytopes for spherical varieties in order to answer a question posed by the second
author.

\noindent 2) The Gelfand-Cetlin approach has been generalized to any reductive group
by the works of Littelmann (\cite{Littelmann})
and Bernstein-Zelevinsky \linebreak
(\cite{B-Z2}). These are called the {\it string polytopes}. Unlike the
case of Gelfand-Cetlin polytopes, in general, the dependence of a string polytope $\Delta(\lambda)$ on
the dominant weight $\lambda$ is not linear.
\end{Rem}

For $G=\GL(n,\c)$ the construction of the polytope  $\tilde
\Delta(A)$ is especially simple and explicit.

Let $G = \GL(n, \c)$. Then $m=n^2$,  $T=(\c^*)^n$, the weight lattice is
$\z^n$ and the Weil group $W = S_n$ acts on $\r^n$ by permuting the
coordinates. The (standard) positive Weil chamber $C$ is the cone
$\{ \lambda = (\lambda_1, \dots , \lambda_n) \mid \lambda_1 \geq
\dots \geq \lambda_n \}$. We say that an $n \times n$ matrix $M=\{
x_{i,j}\}$ with real entries is {\it row-column decreasing} if  its
entries satisfy the inequalities:
\begin{itemize}
\item[(i)] $x_{i,j} \geq x_{i,j+1}$ for $j < n$,
\item[(ii)] $x_{i,j} \geq x_{i+1,j}$ for $i < n$.
\end{itemize}
Let $\mathcal{M}(n)$ be the set all $n\times n$ real row-column
decreasing matrix.

Let $\delta: \r^{n^2} \to \r^n$ be the projection which sends a real $n\times n$
matrix $M = \{x_{i,j}\}$ to its diagonal $\delta(M) =(x_{1,1},\dots, x_{n,n})$.

For the group $\GL(n,\c)$ the polytope $ \Delta_{GC}(\lambda) \times
\Delta_{GC}(\lambda)$ is the polytope in the space of $n \times n$
real matrices, consisting of all $M \in \mathcal{M}(n)$ such that
$\delta(M)= \lambda$. This follows directly from the
original work of Gelfand and Cetlin (\cite{G-C}).

\begin{Def}
For a finite nonempty set $A$ of highest weights for $\GL(n,\c)$,
define the {\it Newton polytope $\Delta_{Newt}(A)$} to be the set of
all matrices $M \in \mathcal{M}(n)$
such that $\delta(M)\in \Delta^+_W(A)$.
\end{Def}

From the defining inequalities of G-C polytopes for $\GL(n,\c)$ one
easily sees that polytope $\Delta_{Newt}(A)$ coincides with the polytope
$\tilde \Delta (A)$. Now the Kazarnovskii theorem for $G=\GL(n\c)$
can be reformulated in terms of the mixed volumes of the above
Newton polytopes:
\begin{Th}[Kazarnovskii theorem for $\GL(n, \c))$] \label{th-Kaz-GL-n}
For finite nonempty subsets $A_1, \ldots, A_{n^2} \subset C \cap \z^n$,
the intersection index $[L_{A_1},\dots,L_{A_{n^2}}]$ is equal to $(n^2)!
V(\Delta_{Newt}(A_1),\dots, \Delta_{Newt}(A_{n^2})).$
\end{Th}

\vspace{.2cm}
{\small
\noindent Askold G. Khovanskii\\Department of
Mathematics, University of Toronto, Toronto, Canada;
Moscow Independent Univarsity; Institute for Systems Analysis, Russian Academy of Sciences.
{\it Email:} {\sf askold@math.utoronto.ca}\\

\noindent Kiumars Kaveh\\Department of Mathematics and Statistics, McMaster University, Hamilton, Canada.
{\it Email:} {\sf kavehk@math.mcmaster.ca}\\
}

\end{document}